\documentclass[12pt,reqno]{amsart}
\usepackage{graphicx}
\usepackage{amsmath}
\usepackage{amsfonts}
\usepackage{enumitem}
\usepackage{amssymb}
\usepackage{amsthm}
\usepackage[english]{babel}
\usepackage{csquotes}
\usepackage{fullpage}
\usepackage{mathtools}
\usepackage[colorlinks=true, pdfstartview=FitH, linkcolor=blue,citecolor=blue, urlcolor=blue]{hyperref}

\newcommand{\Z}{\mathbb{Z}}
\newcommand{\Q}{\mathbb{Q}}
\newcommand{\N}{\mathbb{N}}

\newcommand{\B}{\mathcal{B}}
\renewcommand{\mod}[1]{{\ifmmode\text{\rm\ (mod~$#1$)}\else\discretionary{}{}{\hbox{ }}\rm(mod~$#1$)\fi}}

\newtheorem{theorem}{Theorem}
\newtheorem{proposition}[theorem]{Proposition}
\newtheorem{corollary}[theorem]{Corollary}
\newtheorem{lemma}[theorem]{Lemma}
\theoremstyle{definition}
\newtheorem{definition}[theorem]{Definition}
\newtheorem{remark}[theorem]{Remark}
\numberwithin{equation}{section}
\numberwithin{theorem}{section}

\begin{document}

\title{Multiplicative groups avoiding a fixed group}

\author{Matthias Hannesson}
\address{Department of Mathematics, Vancouver Island University \\ 900 Fifth Street \\ Nanaimo, BC, Canada \ V9R 5S5 }
\email{matthiashannesson@gmail.com}

\author{Greg Martin}
\address{Department of Mathematics, University of British Columbia \\ Room 121, 1984 Mathematics Road \\ Vancouver, BC, Canada \ V6T 1Z2}
\email{gerg@math.ubc.ca}

\begin{abstract}
We know that any finite abelian group~$G$ appears as a subgroup of {\em infinitely many} multiplicative groups~$\mathbb{Z}_n^\times$ (the abelian groups of size $\phi(n)$ that are the multiplicative groups of units in the rings $\mathbb{Z}/n\mathbb{Z}$). It seems to be less well appeciated that~$G$ appears as a subgroup of {\em almost all} multiplicative groups~$\mathbb{Z}_n^\times$. We exhibit an asymptotic formula for the counting function of those integers whose multiplicative group fails to contain a copy of~$G$, for all finite abelian groups~$G$ (other than the trivial one-element group).
\end{abstract}



\maketitle

\section{Introduction}

We know that any finite abelian group~$G$ appears as a subgroup of {\em infinitely many} multiplicative groups~$\Z_n^\times$ (the abelian groups of size $\phi(n)$ that are the multiplicative groups of units in the rings $\Z/n\Z$). The proof of this classic and satisfying problem, which might appear on homework sets for graduate students for example, proceeds as follows: write~$G$ as a direct sum of cyclic groups~$\Z_{m_j}$ with orders $m_1,\dots,m_\ell$; choose primes $p_j\equiv1\mod{m_j}$, so that $\Z_{m_j}$ is a subgroup of $\Z_{p-1} \cong \Z_p^\times$; and then $\Z_{p_1\cdots p_\ell}^\times \cong \Z_{p_1}^\times \times\cdots\times \Z_{p_\ell}^\times$ contains a copy of~$G$.

It seems to be less well appeciated that~$G$ appears as a subgroup of {\em almost all} multiplicative groups~$\Z_n^\times$. In other words, if we define
\begin{equation} \label{SxG def}
S(x;G) = \#\{ n\le x\colon G \nleq \Z_n^\times \} ,
\end{equation}
then $S(x;G) = o(x)$ for any finite abelian group~$G$. The essential reason, from an ``anatomy of integers'' standpoint, is that almost all integers are divisible by any fixed number of primes from any prescribed arithmetic progressions; in particular, almost all integers are divisible by~$\ell$ primes~$p_j$ each congruent to $1\mod{m_j}$.

Whenever analytic number theorists encounter a set of integers whose counting function is $o(x)$, we are motivated to seek better quantitative information about the size of that set. The purpose of this paper is to exhibit an asymptotic formula for $S (x; G)$, the counting function of those integers whose multiplicative group fails to contain a copy of~$G$, for all finite abelian groups~$G$ (other than the trivial one-element group).

The simplest case is when $G \cong \Z_2^{k}$, in which case estimating $S(x;G)$ essentially reduces to counting integers with at most $k-1$ distinct prime factors (as we will show in Section~\ref{proof of main theorem 0}), which is a classical result in analytic number theory.

\begin{theorem} \label{main theorem 0}
For any integer $k \geq 2$,
\begin{equation} \label{main theorem 0 eq}
S(x;\Z_2^k)
=
\frac3{2(k-2)!}
\frac{x (\log\log x)^{k-2}}{\log x}
\biggl(
1 +
O_k\biggl( \frac{1}{\log\log x} \biggr)
\biggr).
\end{equation}
\end{theorem}

For more complicated groups~$G$, the statement of the asymptotic formula depends on some terminology related to the structure of~$G$. Recall that the {\em primary decomposition} of a finite abelian group~$G$ is a representation of~$G$ as the direct sum of cyclic groups whose orders are all powers of primes; this representation is unique up to permutations of the summands. For example, if $G$ is the multiplicative group modulo $11!$, then the Chinese remainder theorem together with the fact that odd prime powers possess primitive roots (and an additional fact about the structure of $\Z_{2^k}^\times$) allow us to see that
\begin{align}
G = \Z_{11!}^\times &= \Z_{2^8\cdot3^4\cdot5^2\cdot7\cdot11}^\times \notag \\
&\cong \Z_{2^8}^\times \times \Z_{3^4}^\times \times \Z_{5^2}^\times \times \Z_7^\times \times \Z_{11}^\times \notag \\
&\cong ( \Z_{2^6} \oplus \Z_2 ) \oplus \Z_{3^3\cdot 2} \oplus \Z_{5^1\cdot 4} \oplus \Z_6 \oplus \Z_{10} \notag \\
&\cong (\Z_{64} \oplus \Z_2) \oplus (\Z_{27} \oplus \Z_2) \oplus (\Z_5 \oplus \Z_4) \oplus (\Z_3 \oplus \Z_2) \oplus (\Z_5 \oplus \Z_2) \label{11!}
\end{align}
is its primary decomposition. In this example there are~$10$ primary summands counted with multiplicity; for our purposes, however, it turns out to be important to gather isomorphic summands together.

\begin{definition} \label{gathered def}
Let~$G$ be a finite abelian group. For any prime power $p^\alpha$ and any positive integer~$k$, we say that $\Z_{p^\alpha}^k$ is a {\em gathered summand} of~$G$ if the primary decomposition of~$G$ contains exactly~$k$ copies of $\Z_{p^\alpha}$, or equivalently if we can write $G \cong \Z_{p^\alpha}^k \times H$ where $\Z_{p^\alpha}$ is not a primary summand of~$H$.

We let $\Gamma(G)$ be the set of all gathered summands of $G$, so that
\begin{equation} \label{generic G decomp}
G \cong \prod_{\Z_{p^\alpha}^k \in \Gamma(G)} \Z_{p^\alpha}^k.
\end{equation}
For example, equation~\eqref{11!} shows that $\Gamma(\Z_{11!}^\times) = \{ \Z_2^4, \Z_3^1, \Z_4^1, \Z_5^2, \Z_{27}^1, \Z_{64}^1 \}$.
\end{definition}

We can now give an asymptotic formula for $S(x;G)$, for general finite abelian groups~$G$, in terms of a sum over the gathered summands of~$G$. The following theorem treats all remaining cases not already covered by Theorem~\ref{main theorem 0}.

\begin{theorem} \label{main theorem 1}
For any finite abelian group~$G \not\cong \Z_2^k$,
\begin{equation*}
S(x;G)
\sim
\sum_{\Z_{p^\alpha}^k \in \Gamma(G)}
K(p^\alpha, k)
\frac{x (\log\log x)^{k -1}}{(\log x)^{1 / \phi(p^\alpha)}},
\end{equation*}
where the constants $K(p^\alpha,k)$ are defined in Definition~\ref{sum constant} below.
\end{theorem}

While the overall structure of this asymptotic formula is easy to parse, we can be even more precise with one additional piece of terminology.

\begin{definition} \label{dominant def}
The set of all possible gathered summands $\Z_{p^\alpha}^k$ can be endowed with a total preorder using the lexicographic ordering of the corresponding ordered pairs $( \phi(p^\alpha), k)$, that is, $\Z_{p^\alpha}^k \preccurlyeq \Z_{q^\beta}^\ell$ when either $\phi(p^\alpha) < \phi(q^\beta)$ or else $\phi(p^\alpha) = \phi(q^\beta)$ and $k \leq \ell$. Equivalently, $\Z_{p^\alpha}^k \preccurlyeq \Z_{q^\beta}^\ell$ precisely when ${x (\log\log x)^{k -1}}/{(\log x)^{1 / \phi(p^\alpha)}} \ll {x (\log\log x)^{\ell -1}}/{(\log x)^{1 / \phi(q^\beta)}}$.
Note that it is possible that simultaneously $\Z_{p^\alpha}^k \preccurlyeq \Z_{q^\beta}^\ell$ and $\Z_{q^\beta}^\ell \preccurlyeq \Z_{p^\alpha}^k$, as the example $\Z_3^k \preccurlyeq \Z_4^k$ and $\Z_4^k \preccurlyeq \Z_3^k$ demonstrates; pairs of prime powers with this property are the sole reason why this preorder is not a total order.

We then call a gathered summand of a finite abelian group~$G$ a {\em dominant summand} if it is a maximal element of $\Gamma(G)$ under this ordering.
\end{definition}

Not surprisingly, the dominant summands provide the dominant contribution to the asymptotic formula for $S(x;G)$.

\begin{theorem} \label{main theorem 2}
Let $G \not \cong \Z_2^k$ be a finite abelian group.
If~$G$ has a unique dominant summand $\Z_{p^\alpha}^k$, then
\begin{equation} \label{one dominant summand}
S(x;G)
=
K(p^\alpha, k)
\frac{x (\log\log x)^{k - 1}}{(\log x)^{1 / \phi(p^\alpha)}}
\biggl( 1 + O_G \biggl( \frac{1}{\log\log x} \biggr) \biggr);
\end{equation}
while if $G$ has two dominant summands $\Z_{p^\alpha}^k$ and $\Z_{q^\beta}^k$, then
\begin{equation} \label{two dominant summands}
S(x;G)
=
\bigl( K(p^\alpha, k) + K(q^\beta,k) \bigr)
\frac{x (\log\log x)^{k - 1}}{(\log x)^{1 / \phi(p^\alpha)}}
\biggl( 1 + O_G \biggl( \frac{1}{\log\log x} \biggr) \biggr).
\end{equation}
In both cases, the leading constants are defined in Definition~\ref{sum constant} below.
\end{theorem}

\begin{remark} \label{dominant remark}
It turns out (as we show in Lemma~\ref{puzzle} below) that~$G$ can have at most two dominant summands, because at most two summands $\Z_{p^\alpha}^k$ and $\Z_{q^\beta}^k$ can be ``equal'' with respect to the total preorder~$\preccurlyeq$ from Definition~\ref{dominant def} (this occurs when $\phi(p^\alpha) = \phi(q^\beta)$, although the values of~$k$ must be equal in such cases). Therefore the statement of Theorem~\ref{main theorem 2} does exhaust all possible cases (when combined with Theorem~\ref{main theorem 0}).
\end{remark}

\begin{remark}
Note that the power of $\log x$ in equation~\eqref{two dominant summands} is the same no matter which dominant summand is chosen, since in this case $1 / \phi(p^\alpha) = 1 / \phi(q ^ \beta)$ by the definition of the ordering~$\preccurlyeq$.
\end{remark}

\begin{remark}
It might seem that we would prove Theorem~\ref{main theorem 1} first, using some sort of divide-and-conquer strategy, and then deduce Theorem~\ref{main theorem 2} from it. However, Theorem~\ref{main theorem 1} is difficult to prove directly in the case where~$G$ has multiple gathered summands corresponding to the same prime~$p$. Consequently, our actual strategy is to establish Theorem~\ref{main theorem 2} as our main goal; it is easy to verify, using only the definitions given so far, that Theorem~\ref{main theorem 1} does follow immediately from Theorem~\ref{main theorem 2}.
\end{remark}

The methods used in this paper are extensions of the methods used by Chang and the second author~\cite{chang} and by Downey and the second author~\cite{downey} when counting the number of multiplicative groups with a particular least invariant factor or a particular $p$-Sylow subgroup, respectively. Section~\ref{tech sec}, which contains many technical lemmas about weighted sums over primes and related counting functions, has the most in common with these prior papers. In Section~\ref{1 GS sec} we establish an asymptotic formula for the counting function $S(x;G)$ from equation~\eqref{SxG def} in the important special case when~$G$ is a single gathered summand~$\Z_{p^\alpha}^k$, in particular proving Theorem~\ref{main theorem 0} therein. Finally, in Section~\ref{general section} we complete the proof of Theorem~\ref{main theorem 2} in full generality.

Throughout this paper,~$p$ and~$q$ are used exclusively to denote primes. Furthermore,~$G$ always denotes a finite abelian group, while~$H$ always denotes a finite abelian $p$-group (a group whose cardinality is a power of~$p$). Furthermore, if~$G$ has the decomposition~\eqref{generic G decomp} as a direct product of its gathered summands, then an implied constant depending on~$G$ means that it may depend on some or all of the quantities $p^\alpha$ and~$k$ appearing in those gathered summands, as well as on the number of gathered summands.

\section{Technical lemmas} \label{tech sec}

In this section we invest in some technical results to prepare ourselves for the main proofs to come, beginning with two preliminary estimates for general functions.  In Section~\ref{downey stuff} we treat several sums over primes in arithmetic progressions that are closely parallel to the methods of~\cite{downey}; in Section~\ref{chang stuff} we import and apply an asymptotic formula from~\cite{chang} for the counting function of integers with prime factors satisfying certain constraints.

To avoid annoying technicalities when~$x$ is small, we use throughout the notation
\begin{equation}
\log_2 x = \log\log \bigl( \max\{x,3\} \bigr).
\end{equation}

\begin{lemma} \label{log estimate}
Let $k \geq 0$ and $\gamma > 0$ be real numbers.
For any real numbers $n\ge2$ and $x\ge n^2$,
\begin{equation*}
\frac{(\log_2 (x/n))^k}{(\log (x/n))^\gamma} =
\frac{(\log_2 x)^k}{(\log x)^\gamma}
\biggl( 1 + O_{ k, \gamma}
\biggl( \frac{\log n}{\log x} \biggr) \biggr)
\ll_{ k, \gamma} \frac{(\log_2 x)^k}{(\log x)^\gamma}.
\end{equation*}
\end{lemma}

\begin{proof}
Since $(\log n)/\log x \le \frac12$, the lemma follows immediately from the approximations
\begin{align*}
\log(x/n) &= \log x \cdot \biggl( 1 - \frac{\log n}{\log x} \biggr) \\
\log\log(x/n) &= \log\log x + \log \biggl( 1 - \frac{\log n}{\log x} \biggr) \\
&= \log\log x + O \biggl( \frac{\log n}{\log x} \biggr) = \log\log x \cdot \biggl( 1 + O \biggl( \frac{\log n}{\log x\log\log x} \biggr) \biggr)
\end{align*}
when $\frac xn\ge3$ (and the lemma is trivial otherwise since all expressions are bounded).
\end{proof}

\begin{lemma} \label{function bound}
Let $k\ge0$ and $\gamma>0$ be real numbers.
Let $f(t)$ be a function satisfying $0\le f(t)\le t$ for $t\ge0$.
Suppose further that $f(t) \ll_{k, \gamma} t (\log_2 t)^k/(\log t)^\gamma$ for $t\ge3$.
Then for any real numbers $x\ge2$ and $m\ge2$,
\begin{align*}
\sum_{\ell = 0}^{\infty}
f(x/m^\ell)
\ll_{k, \gamma}
\frac{x (\log_2 x)^k}{(\log x)^\gamma}.
\end{align*}
\end{lemma}

\begin{proof}
We split the sum into two ranges depending on the relative sizes of $x/m^\ell$ and $\sqrt x$. In the first range, $x/m^\ell \ge m^\ell \ge2$ and thus
\begin{align*}
\sum_{0\le \ell \leq (\log\sqrt x)/\log m} f(x/m^\ell)
&\ll_{k, \gamma}
\sum_{0\le \ell \leq (\log\sqrt x)/\log m} \frac{x (\log_2 (x/m^{\ell}))^k}{m^\ell (\log (x/m^{\ell}))^\gamma} \\
&\ll_{k, \gamma}
\frac{x (\log_2 x)^k}{ (\log x)^\gamma}
\sum_{0\le \ell \leq (\log\sqrt x)/\log m} \frac{1}{m^\ell}
\end{align*}
by Lemma~\ref{log estimate}, and this last geometric series is at most~$2$.
In the second range,
\begin{align*}
\sum_{\ell > (\log\sqrt x) / \log m } f(x/m^\ell)
\leq
\sum_{\ell > (\log\sqrt x) / \log m } \frac{x}{m^\ell}
\ll \frac{x}{m^{(\log\sqrt x) / \log m}} = \sqrt x,
\end{align*}
which completes the proof.
\end{proof}

\subsection{Sums over primes in arithmetic progressions} \label{downey stuff}

The results in this section are direct adaptations of the methods used in the paper~\cite{downey} by Downey and the second author. We give a full proof of the first such lemma, after which we describe the precise relationship between our lemma and the analogous result in~\cite{downey}. For the remaining lemmas, we simply indicate which result in~\cite{downey} is analogous without reproducing the proofs.

We use the notation $\N = \{1,2,\dots\}$ and $\N_0 = \N\cup\{0\}$; in this section $\alpha \in \N_0$ will always be a nonnegative integer. We also use $p^\alpha \parallel n$ to mean that $p^\alpha \mid n$ but $p^{\alpha+1} \nmid n$. Note that when $\alpha=0$, the condition $q\equiv1\mod{p^0}$ is true for all primes~$q$.

\begin{lemma} \label{lower range}
Let $\gamma > 0$ such that $\gamma \notin \N$, let $p$ be prime, and let $\alpha \in \N_0$.
Then for $y \geq 2$,
\begin{equation*}
\sum_{\substack{q \leq \sqrt {y} \\ q \equiv 1 \mod{p^\alpha}}} \frac{1}{q (\log (y/q))^\gamma} = \frac{\log_2 y}{\phi (p^\alpha) (\log y)^\gamma} + O_\gamma \biggl (\frac{1}{(\log y)^\gamma} \biggl ).
\end{equation*}
\end{lemma}

\begin{proof}
We may assume $y\ge9$ since the lemma is trivial for smaller~$y$. We may also absorb the $q=2$ summand (which is relevant in the case $\alpha=0$) into the error term.
If we define
\begin{equation} \label{M(x)}
M(x) = \sum_{\substack{q \leq x \\ q \equiv 1 \mod{p^\alpha}}} \frac{1}{q},
\end{equation}
it then follows by a well-known result (see for example~\cite[Corollary 4.12]{montgomery}) that there exists a constant $c_{p^\alpha}$ such that $M(x) = (\log_2 x) /\phi(p^\alpha) + c_{p^\alpha} + O(1/\log x)$ for $x \geq 3$ (note that $\log_2 x$ and $\log\log x$ are synonymous in this range).
We can then define
\begin{equation} \label{R(x)}
R(x) = M(x) - (\log_2 x) /\phi(p^\alpha) - c_{p^\alpha},
\end{equation}
so that $R(x) \ll 1/\log x$. It then follows that
\begin{align*}
\sum_{\substack{3\le q \leq \sqrt {y} \\ q \equiv 1 \mod{p^\alpha}}} \frac{1}{q (\log (y/q))^\gamma}
&= \int_{3}^{\sqrt{y}} \frac{1}{(\log (y/u))^\gamma} \,dM(u) \\
&= \int_{3}^{\sqrt{y}} \frac{1}{ (\log (y/u))^\gamma} \,d  \bigl( (\log_2 u) /\phi(p^\alpha) + c_{p^\alpha} + R(u) \bigr) \\
&= \frac{1}{\phi(p^\alpha)} \int_{3}^{\sqrt{y}} \frac{1}{ (\log (y/u))^\gamma} \frac{du}{u \log u} + \int_{3}^{\sqrt{y}} \frac{1}{(\log (y/u))^\gamma} \,dR(u) \\
&= \frac{1}{\phi(p^\alpha)} \biggl ( \frac{\log_2 y}{(\log y)^\gamma} + O_\gamma \biggl(\frac{1}{(\log y)^\gamma}\biggr) \biggr ) + \int_{3}^{\sqrt{y}} \frac{1}{(\log (y/u))^\gamma} \,dR(u),
\end{align*}
where we used~\cite[Lemma~2.11]{downey} to estimate the first integral in the second-to-last line.
Applying integration by parts to the remaining integral yields
\begin{align*}
\int_{3}^{\sqrt{y}} \frac{1}{(\log (y/u))^\gamma} \,dR(u)
&= \frac{R(u)}{(\log (y/u))^\gamma} \biggr |_3^{\sqrt{y}} - \int_3^{\sqrt{y}} R(u) \frac{d}{du} \biggl ( \frac{1}{(\log (y/u))^\gamma} \biggr ) \,du \\
& = \frac{R(\sqrt{y})}{(\log (\sqrt{y}))^\gamma} - \frac{R(3)}{(\log (y/3))^\gamma}
- \int_3^{\sqrt{y}} \frac{\gamma R(u)}{u (\log (y/u))^{\gamma+1}} \,du \\
& \ll_\gamma \frac{1/\log y}{(\log y)^\gamma} + \frac{1}{(\log y)^\gamma}
+ \int_3^{\sqrt{y}} \frac{1 / \log u}{u (\log (y/u))^{\gamma+1} } \,du \\
& \ll_\gamma \frac{1}{(\log y)^\gamma} + \frac{1}{(\log y)^{\gamma+1}} \int_3^{\sqrt{y}} \frac{1}{u \log u} \,du \\
& \ll \frac{1}{(\log y)^\gamma} + \frac{\log_2 y}{(\log y)^{\gamma +1}},
\end{align*}
which completes the proof.
\end{proof}

\begin{remark}
Lemma~\ref{lower range} is virtually the same as~\cite[Lemma~2.12]{downey}, except that our sum runs over primes~$q$ with $q\equiv1\mod{p^\alpha}$ whereas their sum runs over primes~$q$ with $p^\alpha\parallel(q-1)$.
Note that our constant $1/\phi(p^\alpha)$ is the relative density of $\{ q\colon q\equiv1\mod{p^\alpha} \}$ within the primes, while their constant $1/p^\alpha$ is the relative density of $\{ q\colon p^\alpha\parallel(q-1) \}$.
Indeed, we can recover their result immediately from ours, since
\[
\{ q\colon p^\alpha\parallel(q-1) \} = \{ q\colon q\equiv1\mod{p^\alpha} \} \setminus \{ q\colon q\equiv1\mod{p^{\alpha+1}} \}.
\]
Unfortunately, we cannot directly prove our result from theirs, since
\[
\{ q\colon q\equiv1\mod{p^\alpha} \} = \{ q\colon p^\alpha\parallel(q-1) \} \cup \{ q\colon p^{\alpha+1}\parallel(q-1) \} \cup \{ q\colon p^{\alpha+2}\parallel(q-1) \} \cup \cdots,
\]
and handling the resulting infinite sum leads to worse error terms. Nevertheless, other than the fact that we start from equation~\eqref{R(x)} while they start from~\cite[equation~(2.6)]{downey}, the two proofs are identical (up to notational choices).
\end{remark}

\begin{lemma} \label{upper range}
Let $\gamma > 0$ such that $\gamma \notin \N$, let $p$ be prime, and let $\alpha \in \N_0$. Then for $y \geq 9$,
\begin{equation*}
\sum_{\substack{\sqrt{y} < q \leq y/2 \\ q \equiv 1 \mod{p^\alpha}}} \frac{1}{q (\log (y/q))^\gamma} \ll_\gamma \frac{1}{(\log y)^{\min\{1, \gamma\}} }.
\end{equation*}
\end{lemma}

\begin{proof}
This formula is directly analogous to~\cite[Lemma~2.14]{downey}, except that our sum runs over primes~$q$ with $q\equiv1\mod{p^\alpha}$ whereas their sum runs over primes~$q$ with $p^\alpha\parallel(q-1)$.
(We have also chosen to let~$q$ range up to $\frac y2$ in this paper, while they used an upper bound of $\frac y3$.)
The proof is again essentially identical, other than starting from equation~\eqref{R(x)} and using Lemma~\ref{lower range} while they start from~\cite[equation~(2.6)]{downey} and use~\cite[Lemma~2.12]{downey}.
\end{proof}

\begin{lemma} \label{logsum}
Let $k \geq 0$ and $\gamma > 0$ such that $\gamma \notin \N$, let $p$ be prime, and let $\alpha \in \N_0$. Then for $y \geq 2$,
\begin{equation*}
\sum_{\substack{q \leq y/2 \\ q \equiv 1 \mod{p^\alpha}}} \frac{(\log_2 (y/q))^k}{q (\log (y/q))^\gamma}
=
\frac{(\log_2 y)^{k+1}}{\phi (p^\alpha) (\log y)^\gamma} + O_{k, \gamma} \biggl (\frac{(\log_2 y)^k}{(\log y)^{\min\{1, \gamma\}}} \biggl ).
\end{equation*}
\end{lemma}

\begin{proof}
This formula is directly analogous to~\cite[Lemma~2.18]{downey}, except that again our sum runs over primes~$q$ with $q\equiv1\mod{p^\alpha}$ whereas their sum runs over primes~$q$ with $p^\alpha\parallel(q-1)$. The proof is again otherwise identical.
\end{proof}

\begin{lemma} \label{logsum bound 1}
Let $k \geq 0$ and $ \gamma > 0$ such that $\gamma \notin \N$. Then for $y \geq 2$,
\begin{equation*}
\sum_{q \leq y/2} \frac{(\log_2 (y/q))^k}{q ( \log (y/q))^\gamma} \ll_{k, \gamma} \frac{(\log_2 y)^{k+1}}{(\log y)^{\min \{\gamma, 1 \}}}.
\end{equation*}
\end{lemma}

\begin{proof}
This bound is an immediate consequence of the $\alpha = 0$ case of Lemma~\ref{logsum}.
\end{proof}

\begin{lemma} \label{logsum bound 2}
Let $k \geq 0$. Then for $y \geq 2$,
\begin{equation*}
\sum_{q \leq y/2} \frac{(\log_2 (y/q))^k}{q (\log (y/q))} \ll_k \frac{(\log_2 y)^{k+1}}{\log y}.
\end{equation*}
\end{lemma}

\begin{proof}
The lemma follows by splitting the range of summation and noting that
\begin{align*}
\sum_{q \leq \sqrt{y}} \frac{(\log_2 (y/q))^k}{q (\log (y/q))}
&\ll \frac{(\log_2 y)^k}{\log y} \sum_{q \leq \sqrt{y}} \frac{1}{q}
\ll \frac{(\log_2 y)^{k+1}}{\log y} \\
\sum_{\sqrt{y} < q \leq y/2} \frac{(\log_2 (y/q))^k}{q (\log (y/q))}
&\ll (\log_2 y)^k \sum_{\sqrt{y} < q \leq y/2} \frac{1}{q \log (y/q)}
\ll \frac{(\log_2 y)^{k+1}}{\log y},
\end{align*}
where this last sum is bounded using~\cite[Lemma 2.15]{downey}.
\end{proof}

\begin{lemma} \label{nested logsum}
Let $j \in \N$ and $0 < \gamma < 2$. Then for $y \ge 2$,
\begin{equation*}
\sum_{\substack{p_1 \leq y/2 }} \frac{1}{p_1} \cdots
\sum_{\substack{ p_{j} \leq y/2 p_1 \cdots p_{j-1} }} \frac{1}{p_j} \frac{1}{(\log(y/p_1 \cdots p_j))^\gamma}
\ll_{j, \gamma}
\frac{(\log_2 y)^j}{(\log y)^{\min \{\gamma, 1 \}}}.
\end{equation*}
\end{lemma}

\begin{proof}
The lemma follows by an inductive application of Lemma~\ref{logsum bound 1}, or Lemma~\ref{logsum bound 2} in the case that $\gamma = 1$, since for any $ i \leq j$ we have
\begin{equation*}
\sum_{p_i \leq y/2 p_1 \cdots p_{i-1}}
\frac{1}{p_i}
\frac{(\log_2 (y/p_1 \cdots p_i))^{j-i}}{(\log (y/p_1 \cdots p_i))^{\gamma}}
\ll_{j, \gamma} \frac{ (\log_2 (y/p_1 \cdots p_{i-1}))^{j-(i-1)}}{(\log (y/p_1 \cdots p_{i-1}))^{\min \{\gamma, 1\}}}. \qedhere
\end{equation*}
\end{proof}

\begin{lemma} \label{nested prime sum}
Let $j \in \N$. Then for $y \geq 3$,
\begin{equation*}
\sum_{\substack{p_1 \leq y/2 }} \cdots
\sum_{\substack{ p_{j-1} \leq y/2 p_1 \cdots p_{j-2} }}
\sum_{\substack{ y/2p_1 \cdots p_{j-1} < p_j \leq y/p_1 \cdots p_{j-1} }}
\sum_{n \leq y/p_1 \cdots p_j} 1
\ll_{j} \frac{y (\log_2 y)^{j-1}}{\log y}.
\end{equation*}
\end{lemma}

\begin{proof}
Since $y/2p_1 \cdots p_{j-1} < p_j$, it follows that the inner sum can have at most one term.
Therefore
\begin{align*}
\sum_{\substack{p_1 \leq y/2 }} \cdots
\sum_{\substack{ y/2p_1 \cdots p_{j-1} < p_j \leq y/p_1 \cdots p_{j-1} }}
\sum_{n \leq y/p_1 \cdots p_j} 1
\le \sum_{\substack{p_1 \leq y/2 }} \cdots
\sum_{\substack{ p_{j-1} \leq y/2 p_1 \cdots p_{j-2} }}
\sum_{\substack{ y/2p_1 \cdots p_{j-1} < p_j \leq y/p_1 \cdots p_{j-1} }} 1& \\
\leq \sum_{\substack{p_1 \leq y/2 }} \cdots
\sum_{\substack{ p_{j-1} \leq y/2 p_1 \cdots p_{j-2} }} \pi (y/p_1 \cdots p_{j-1})
\ll \sum_{\substack{p_1 \leq y/2 }} \cdots
\sum_{\substack{ p_{j-1} \leq y/2 p_1 \cdots p_{j-2} }}
\frac{y/ p_1 \cdots p_{j-1}}{\log (y/ p_1 \cdots p_{j-1})},&
\end{align*}
so the lemma follows by Lemma~\ref{nested logsum}.
\end{proof}

\begin{lemma} \label{big logsum}
Let $k \geq 0$ and $\gamma > 0$ such that $\gamma \notin \N$, let $p_1, \ldots , p_j$ be distinct primes, let $p$ be prime and let $\alpha \in \N$.
Then for $y \ge 2$,
\begin{align*}
\sum_{\substack{q \leq y/2 \\ q \equiv 1 \mod{p^\alpha} \\ q \nmid p_1 \cdots p_j}}
\frac{q + O(1)}{q^2}
\biggl (
\frac{(\log_2 (y/q))^k}{(\log (y/q))^\gamma}
+ O \biggl (
\frac{(\log_2 (y/q))^{k-1}}{(\log (y/q))^{\min{\{ \gamma , 1 \}}}}
\biggr ) \biggr)\\
= \frac{(\log_2 y)^{k+1}}{\phi(p^\alpha)(\log y)^\gamma}
+ O_{j,k,\gamma} \biggl (
\frac{(\log_2 y)^{k}}{(\log y))^{\min{\{ \gamma , 1 \}}}}
\biggr ).
\end{align*}
\end{lemma}

\begin{proof}
This formula is again directly analogous to~\cite[Proposition~2.20]{downey}, except that our sum runs over primes~$q$ with $q\equiv1\mod{p^\alpha}$ whereas their sum runs over primes~$q$ with $p^\alpha\parallel(q-1)$.
\end{proof}

\subsection{Integers with constrained prime factors} \label{chang stuff}

We will need asymptotic formulas and estimates for the number of integers with particular constraints (most notably congruence restrictions) on their prime factors. We begin by defining a complicated constant occurring in the main term of such asymptotic formulas.

\begin{definition} \label{zeta constant}
Let $\B$ be a set of~$\tau$ distinct reduced residue classes modulo~$d$, and define
\begin{equation*}
\delta( \B ) =
\lim_{s \to 1^+} \zeta (s) ^{-\tau / \phi(d)}
\prod_{q \in \B}
\biggl (
1 - \frac{1}{q^s}
\biggr ) ^{-1}.
\end{equation*}
The limit exists because the function on the right-hand side has an analytic continuation to a neighbourhood of $s=1$ (see for example~\cite[Lemma~3.2]{chang}).
\end{definition}

\begin{lemma} \label{restricted prime factor sum}
Let $\B$ be a set of~$\tau$ distinct reduced residue classes modulo~$d$, and let $p_1, \ldots, p_m$ be distinct primes not belonging to~$\B$.
Then for $x \geq 2$,
\begin{equation*}
\sum_{\substack{n \leq x \\  (q \mid n \textrm{ and } q \notin \B) \implies q \mid p_1 \cdots p_m}} 1
=
\frac{\delta(\B)}{\Gamma (\tau / \phi (d))} \frac{p_1 \cdots p_m}{\phi(p_1 \cdots p_m)} \frac{x}{(\log x)^{1 - \tau/\phi(d)}}
+ O_{d, m} \biggl ( \frac{x}{(\log x)^{2 - \tau/\phi(d)}} \biggl ).
\end{equation*}
\end{lemma}

\begin{proof}
Let $\mathcal{I} = \{ p_1, \ldots, p_m \}$ and $\mathcal{R} = \varnothing$, so that we are counting integers all of whose prime factors belong to $\B \cup \mathcal{I} \backslash \mathcal{R}$.
Since we are allowing our error term depend on~$d$ and~$m$, the statement of~\cite[Theorem 3.6]{chang} simplifies to
\begin{align*}
\sum_{\substack{n \leq x \\  (q \mid n \textrm{ and } q \notin \B) \implies q \mid p_1 \cdots p_m}} 1
&=
\#\{ n \leq x \colon   q \mid n  \implies q \in \B \cup \mathcal{I} \backslash \mathcal{R}  \} \\
&= \frac{x}{(\log x)^{1 - \tau/\phi(d)}} \Bigg ( \frac{\delta(\B)}{\Gamma (\tau / \phi (d))} \prod_{i = 1}^m \biggl( 1-\frac{1}{p_i} \biggl )^{-1}
{}+ O_{d, m} \biggl ( \frac{1}{\log x} \biggl ) \Bigg ),
\end{align*}
which is equivalent to the statement of the lemma. (Their Theorem~3.6 includes the constraint $\log x \gg d^{1/2} \log^2 d$; however, since we allow our $O$-constant to depend on~$d$, we may extend that range all the way down to $x\ge3$.)
\end{proof}

\begin{definition}
Let~$\B$ be any set of positive integers. For any~$m\in\N_0$, define
\begin{equation*}
D_m(x;\B) = \#\bigl\{ N\le x\colon N \text{ has exactly $m$ distinct prime factors that are not in } \B \bigr\}.
\end{equation*}
Note that in particular
\begin{equation} \label{special 0 case}
D_0(x; \B) = \sum_{\substack{n \leq x \\ {q \mid n \implies q \in \B}}} 1.
\end{equation}
\end{definition}

\begin{lemma} \label{most of C(x,T)}
Let~$\B$ be the union of~$\tau$ distinct reduced residue classes modulo~$d$.
Then
\begin{multline} \label{most multline}
D_m(x;\B) =
\frac1{m!} \sum_{\substack{p_1 \leq x/2 \\ p_1 \notin \B}}
\sum_{\substack{p_2 \leq x/2p_1 \\ p_2 \notin \B \\ p_2 \neq p_1}}
\cdots
\sum_{\substack{ p_m \leq x /2 p_1 \cdots p_{m-1} \\ p_m \notin \B \\ p_m \nmid p_1 \cdots p_{m-1}}} \sum_{\substack{n \leq x/p_1 \cdots p_m \\ (q \mid n \textrm{ and } q \notin \B) \implies q \mid p_1 \cdots p_m}} 1 \\
+ O_{d,m} \biggl( \frac{x (\log_2 x)^{m-1}}{(\log x)^{1 - \tau / \phi(d)}} \biggr).
\end{multline}
\end{lemma}

\begin{proof}
When $m=0$ the claim is immediate from equation~\eqref{special 0 case}, so we may assume $m\ge1$.
We first claim that
\begin{equation} \label{hopeful edit}
D_m(x;\B) =
\frac1{m!} \sum_{\substack{p_1 \leq x \\ p_1 \notin \B}}
\sum_{\substack{p_2 \leq x/p_1 \\ p_2 \notin \B \\ p_2 \neq p_1}}
\cdots
\sum_{\substack{ p_j \leq x / p_1 \cdots p_{j-1} \\ p_j \notin \B \\ p_j \nmid p_1 \cdots p_{j-1}}} \sum_{\substack{n \leq x/p_1 \cdots p_j \\ (q \mid n \textrm{ and } q \notin \B) \implies q \mid p_1 \cdots p_j}} 1.
\end{equation}
To see this, for any~$N$ counted by $D_m(x;\B)$, let $p_1,\dots,p_m$ be the distinct prime factors outside~$\B$ that divide~$N$, and set $n=N/p_1\cdots p_m$; then~$n$ is counted by the innermost sum in equation~\eqref{hopeful edit}, while the factor of $1/m!$ comes from the fact that each distinct set of primes $\{ p_1, ... , p_m \}$ outside of~$\B$ will be counted $m!$ times by the multiple sum.

In equation~\eqref{hopeful edit}, if $p_j > x/2p_1 \cdots p_{j-1}$ for any $1 \leq j \leq m-1$ then the the sum over $p_{j+1} \leq x/p_1 \cdots p_i$ will be empty; consequently we may adjust those ranges of summation to $p_j\le x/2p_1 \cdots p_{j-1}$. We also split the sum over~$p_m$ into two ranges at $x/2p_1\cdots p_{m-1}$, which yields $D_m(x;\B) = \frac1{m!}(J_1 + J_2)$ where we have defined
\begin{align*}
J_1 &=
\sum_{\substack{p_1 \leq x/2 \\ p_1 \notin \B}}
\sum_{\substack{p_2 \leq x/2p_1 \\ p_2 \notin \B \\ p_2 \neq p_1}} \cdots \sum_{\substack{p_m \leq x/2p_1\cdots p_{m-1} \\ p_m \equiv \notin \B \\ p_m \nmid p_1 \cdots p_{m-1}}}
\sum_{\substack{n \leq x/p_1\cdots p_m \\ {(q \mid n \textrm{ and } q \nmid p_1\cdots p_m) \implies q \in \B}}} 1 \\
J_2 &= \sum_{\substack{p_1 \leq x/2 \\ p_1 \notin \B}}
\sum_{\substack{p_2 \leq x/2p_1 \\ p_2 \notin \B \\ p_2 \neq p_1}} \cdots \sum_{\substack{x/2p_1 \cdots p_{m-1} < p_m \leq x/p_1 \cdots p_{m-1} \\ p_m \notin \B \\ p_m \nmid p_1 \cdots p_{m-1}}}
\sum_{\substack{n \leq x/p_1\cdots p_m \\ {(q \mid n \textrm{ and } q \nmid p_1\cdots p_m) \implies q \in \B}}} 1.
\end{align*}
For~$J_2$, we may simply ignore the congruence and divisibility restrictions on the~$p_j$ and on~$n$ so as to apply Lemma~\ref{nested prime sum}, yielding the estimate $J_2 \ll_m x (\log_2 x)^{m-1} / \log x$ that establishes the lemma.
\end{proof}

\begin{lemma} \label{C(x,T) new}
Let~$\B$ be the union of~$\tau$ distinct reduced residue classes modulo~$d$. For any~$k\in\N_0$, the number of integers up to~$x$ that have at most~$k$ distinct prime factors that are not in~$\B$ is
$\ll_{d, k} {x (\log_2 x)^k}/{(\log x)^{1-\tau/\phi(d)}}$.
\end{lemma}

\begin{proof}
For any $0\le m\le k$, we invoke Lemma~\ref{most of C(x,T)} and then apply Lemma~\ref{restricted prime factor sum} to the innermost sum in equation~\eqref{most multline} to obtain
\begin{align*}
D_m(x;\B)
\ll_{d, m}
\sum_{\substack{p_1 \leq x/2 \\ p_1 \notin \B}}
\cdots
\sum_{\substack{ p_m \leq x /2 p_1 \cdots p_{m-1} \\ p_m \notin \B \\ p_m \nmid p_1 \cdots p_{m-1}}}
\frac{x /p_1 \cdots p_m}{(\log (x/p_1 \cdots p_m))^{1 - \tau / \phi(d)}} + \frac{x (\log_2 x)^{m-1}}{(\log x)^{1 - \tau / \phi(d)}}.
\end{align*}
We may now ignore the restrictions that the~$p_i$ are distinct and not in~$\B$ so as to apply Lemma~\ref{nested logsum}, from which we deduce that $D_m(x;\B) \ll_{d,m} x (\log_2 x)^j / (\log x)^{1 - c / \phi(d)}$; the lemma follows upon summing over $0\le m\le k$.
\end{proof}

\begin{definition} \label{D m x p alpha def}
For any~$m\in\N_0$, define
\begin{multline*}
D_m(x;p^\alpha) = \#\bigl\{ N\le x\colon p\nmid N, \\
\text{$N$ has exactly $m$ distinct prime factors congruent to~$1$ modulo $p^\alpha$} \bigr\}.
\end{multline*}
Note that this is almost the same quantity as $D_m(x;\B_{p^\alpha})$, where
\begin{equation} \label{B p alpha def}
\B_{p^\alpha} = \{ c \in \Z_{p^\alpha}^\times \colon c \not\equiv 1 \mod{p^\alpha} \}
\end{equation}
is the union of $\phi(p^\alpha)-1$ reduced residue classes modulo~$p^\alpha$; however, $D_m(x;\B_{p^\alpha})$ counts some integers divisible by~$p$ while $D_m(x;p^\alpha)$ does not. Nevertheless, the proof of Lemma~\ref{most of C(x,T)} goes through virtually unchanged to show that
\begin{multline} \label{most multline p}
D_m(x;p^\alpha) =
\frac1{m!} \sum_{\substack{p_1 \leq x/2 \\ p_1\equiv1\mod{p^\alpha}}}
\sum_{\substack{p_2 \leq x/2p_1 \\ p_2\equiv1\mod{p^\alpha} \\ p_2 \neq p_1}}
\cdots
\sum_{\substack{ p_m \leq x /2 p_1 \cdots p_{m-1} \\ p_m\equiv1\mod{p^\alpha} \\ p_m \nmid p_1 \cdots p_{m-1}}} \sum_{\substack{n \leq x/p_1 \cdots p_m \\ (q \mid n \textrm{ and } q\notin\B_{p^\alpha}) \implies q \mid p_1 \cdots p_m}} 1 \\
+ O_{d,m} \biggl( \frac{x (\log_2 x)^{m-1}}{(\log x)^{1 - \tau / \phi(d)}} \biggr).
\end{multline}
\end{definition}

\begin{lemma} \label{D-sum estimate new}
For any prime power $p^\alpha \geq 3$ and any~$m\in\N_0$,
\begin{align*}
D_m(x;p^\alpha)
=
\frac{\delta(\B_{p^\alpha})/m!}{\Gamma (1 - 1/\phi(p^\alpha))} \frac{x(\log_2 x)^m}{\phi(p^\alpha)^m (\log x)^{1/\phi(p^\alpha)}}
+
O_{p, \alpha , m}
\biggl (
\frac{x(\log_2 x)^{m-1}}{ (\log x)^{1/\phi(p^\alpha)}}
\biggr ).
\end{align*}
In particular, $D_m(x;p^\alpha) \ll_{p, \alpha , m} {x(\log_2 x)^m}/{ (\log x)^{1/\phi(p^\alpha)}}$.
\end{lemma}

\begin{proof}
Let~$J_1$ equal the multiple sum in equation~\eqref{most multline p}.
We apply Lemma~\ref{restricted prime factor sum} with $d=p^\alpha$ and $\B=\B_{p^\alpha}$ from equation~\eqref{B p alpha def}, so that $\tau=\phi(p^\alpha)-1$, to the innermost sum to obtain
\begin{align}
J_1 = \frac1{m!}
\sum_{\substack{p_1 \leq x/2 \\ p_1 \equiv 1 \mod{p^\alpha}}}  \cdots
\sum_{\substack{p_m \leq x/2p_1\cdots p_{m-1} \\ p_m \equiv 1 \mod{p^\alpha} \\ p_m \nmid p_1 \cdots p_{m-1}}}
\Bigg (
\frac{\delta(\B_{p^\alpha})}{\Gamma (1 - 1/\phi(p^\alpha))}
\frac{x/p_1 \cdots p_m}{(\log (x/p_1 \cdots p_m))^{1/\phi(p^\alpha)}}
\frac{p_1\cdots p_m}{\phi(p_1\cdots p_m)} \notag \\
{}+ O_{p, \alpha, m} \biggl ( \frac{x/p_1 \cdots p_m}{(\log (x/p_1 \cdots p_m))^{1 + 1/\phi(p^\alpha)}} \biggl )
\Bigg ) \notag \\
= \frac{\delta(\B_{p^\alpha}) x/m!}{\Gamma (1 - 1/\phi(p^\alpha))}
\sum_{\substack{p_1 \leq x/2 \\ p_1 \equiv 1 \mod{p^\alpha}}}
\frac{p_1 + O(1)}{p_1^2}
\cdots
\sum_{\substack{p_m \leq x/2p_1\cdots p_{m-1} \\ p_m \equiv 1 \mod{p^\alpha} \\ p_m \nmid p_1 \cdots p_{m-1}}}
\frac{p_m + O(1)}{p_m^2}
\frac{1}{(\log (x/p_1 \cdots p_m))^{1/\phi(p^\alpha)}} \notag \\
{}+ O_{p, \alpha, m}
\Bigg(
\sum_{\substack{p_1 \leq x/2 \\ p_1 \equiv 1 \mod{p^\alpha}}}
\cdots
\sum_{\substack{p_m \leq x/2p_1\cdots p_{m-1} \\ p_m \equiv 1 \mod{p^\alpha} \\ p_m \nmid p_1 \cdots p_{m-1}}}
\frac{x/p_1 \cdots p_m}{(\log (x/p_1 \cdots p_m))^{1+1/\phi(p^\alpha)}}
\Bigg ) .
\label{J2 main term new}
\end{align}
In the error term, we may again ignore the congruence and distinctness conditions on the~$p_j$ so as to apply Lemma~\ref{nested logsum}, which yields the acceptable error $\ll_{p,\alpha,m} x (\log_2 x)^{m-1} / \log x$.

Regarding the main term on the right-hand side of equation~\eqref{J2 main term new}, Lemma~\ref{big logsum} tells us that for any $i \leq m$,
\begin{align*}
\frac{1}{\phi (p^\alpha)^{m-i}} \sum_{\substack{p_i \leq x/2p_1 \cdots p_{i-1} \\ p_i \equiv 1\mod{p^\alpha}}}
\frac{p_i + O(1)}{p_i^2}
\Bigg (
\frac{(\log_2 (x/p_1 \cdots p_i))^{m-i}}{(\log (x/p_1 \cdots p_i)^{1/\phi(p^\alpha)}}& \\
+ O_{p, \alpha, m}
\biggl ( &
\frac{(\log_2 (x/p_1 \cdots p_i))^{m-i-1}}{(\log (x/p_1 \cdots p_i)^{1/\phi(p^\alpha)}}
\biggr )
\Bigg) \\
=
\frac{1}{\phi (p^\alpha)^{m-(i-1)}}
\frac{(\log_2 (x/p_1 \cdots p_{i-1}))^{m-(i-1)}}{(\log (x/p_1 \cdots p_{i-1})^{1/\phi(p^\alpha)}}
+ O_{p, \alpha, m}
\biggl (&
\frac{(\log_2 (x/p_1 \cdots p_{i-1}))^{m-i}}{(\log (x/p_1 \cdots p_{i-1})^{1/\phi(p^\alpha)}}
\biggr ).
\end{align*}
This observation allows us to iteratively evaluate the innermost sum one at a time, and it follows inductively in this way that
\begin{align*}
\sum_{\substack{p_1 \leq x/2 \\ p_1 \equiv 1 \mod{p^\alpha}}}
\frac{p_1 + O(1)}{p_1^2}
\cdots
\sum_{\substack{p_m \leq x/2p_1\cdots p_{m-1} \\ p_m \equiv 1 \mod{p^\alpha} \\ p_m \nmid p_1 \cdots p_{m-1}}}
\frac{p_m + O(1)}{p_m^2}
\frac{1}{(\log (x/p_1 \cdots p_m))^{1/\phi(p^\alpha)}} \\
=
\frac{(\log_2 x)^m}{\phi(p^\alpha)^m (\log x)^{1/\phi(p^\alpha)}}
+
O_{p, \alpha , m}
\biggl (
\frac{(\log_2 x)^{m-1}}{ (\log x)^{1/\phi(p^\alpha)}}
\biggr ),
\end{align*}
which completes the proof of the lemma when inserted back into equation~\eqref{J2 main term new}.
\end{proof}

\section{Individual gathered summands} \label{1 GS sec}

The estimation of $S(x;\Z_{p^\alpha}^k)$ will be crucial to the estimation of $S(x; G)$ for general finite abelian groups~$G$. It will be helpful to group the integers counted by $S(x;\Z_{p^\alpha}^k)$ according to the power of~$p$ dividing them.

\begin{definition} \label{S_l definition}
For any prime power $p^\alpha$ and any integers $k\ge1$ and $\ell\ge0$, define
\[
S_\ell(x;\Z_{p^\alpha}^k) = \#\bigl\{ n\le x\colon \Z_{p^\alpha}^k \nleq \Z_n^\times,\, p^\ell \parallel n \bigr\}.
\]
This quantity refines $S(x;\Z_{p^\alpha}^k)$, and indeed a comparison to Definition~\ref{SxG def} reveals that
\begin{equation} \label{group by ell}
S(x;\Z_{p^\alpha}^k) = \sum_{\ell=0}^{\infty} S_\ell(x;\Z_{p^\alpha}^k).
\end{equation}
\end{definition}

Fortunately, every quantity $S_\ell(x;\Z_{p^\alpha}^k)$ can be expressed in terms of $S_0(y;\Z_{p^\alpha}^j)$ for various values of~$y$ and~$j$, as we now demonstrate.

\begin{lemma} \label{Sell to S0 lemma}
For any prime power $p^\alpha\ge3$ and any integers $k\ge1$ and $\ell\ge0$,
\[
S_\ell(x;\Z_{p^\alpha}^k) = \begin{cases}
S_0 (x/p^\ell;\Z_{p^\alpha}^k), &\text{if $p$ is odd and } \ell\le\alpha, \\
S_0 (x/p^\ell;\Z_{p^\alpha}^{k-1}), &\text{if $p$ is odd and } \ell\ge\alpha+1, \\
S_0 (x/2^\ell;\Z_{2^\alpha}^k), &\text{if $p=2$ and $\alpha \geq 2$ and } \ell\le\alpha+1, \\
S_0 (x/2^\ell;\Z_{2^\alpha}^{k-1}), &\text{if $p=2$ and $\alpha \geq 2$ and } \ell\ge\alpha+2.
\end{cases}
\]
Furthermore,
$S_1 (x;\Z_2^k) = S_0 (x/2;\Z_2^k)$ and $S_2 (x;\Z_2^k) = S_0 (x/4,\Z_2^{k-1})$, while $S_\ell (x;\Z_2^k) = S_0 (x/2^\ell,\Z_2^{k-2})$ for $\ell\ge3$.
\end{lemma}

\begin{proof}
Every integer counted by $S_\ell(x;\Z_{p^\alpha}^k)$ can be written as $n=p^\ell r$ where $r\le x/p^\ell$ and $p\nmid r$; as a result, $\Z_n^\times \cong \Z_{p^\ell}^\times \times \Z_r^\times$.

If~$p$ is odd, then $\Z_{p^\ell}^\times \cong \Z_{p-1} \times \Z_{p^{\ell-1}}$, which contains a copy of $\Z_{p^\alpha}$ if and only if $\ell\ge\alpha+1$. Therefore $\Z_r^\times$ cannot contain~$k$ copies of $\Z_{p^\alpha}$ if $\ell\le\alpha$, and cannot contain $k-1$ copies of $\Z_{p^\alpha}$ if $\ell\ge\alpha+1$.

If $p=2$, then $\Z_{2^1}^\times \cong \Z_1$ and $\Z_{2^2}^\times \cong \Z_2$, while $\Z_{2^\ell}^\times \cong \Z_2 \times \Z_{2^{\ell-2}}$ for $\ell\ge3$. When $\alpha\ge2$, these groups contain a copy of $\Z_{2^\alpha}$ if and only if $\ell\ge\alpha+2$. Therefore $\Z_r^\times$ cannot contain~$k$ copies of $\Z_{2^\alpha}$ if $\ell\le\alpha+1$, and cannot contain $k-1$ copies of $\Z_{2^\alpha}$ if $\ell\ge\alpha+2$. The identities for $S_\ell(x;\Z_2^k)$ follow similarly by noting that $\Z_{2^\ell}^\times$ contains~$0$, $1$, or~$2$ copies of $\Z_2$ depending on whether $\ell=1$, $\ell=2$, or $\ell\ge3$.
\end{proof}

Thanks to Lemma~\ref{Sell to S0 lemma}, we can concentrate solely on $S_0(x;\Z_{p^\alpha}^k)$ to understand the behaviour of $S(x;\Z_{p^\alpha}^k)$. It turns out that the case $p^\alpha=2$ is special, but also straightforward, so we handle that case first.

\subsection{The special case of $\Z_2^k$}
\label{proof of main theorem 0}

In this section, we prove Theorem~\ref{main theorem 0}.
The reason that the case $p^\alpha=2$ is different from other cases is that {\em every} odd prime dividing~$n$ produces a primary factor~$\Z_2$ in the multiplicative group~$\Z_n^\times$, and more such factors are produced if $4\mid n$.
In particular, it follows immediately that if $\Z_2 \nleq \Z_n^\times$ then either $n=1$ or $n = 2$, and therefore $S(x;\Z_2) = 2$ for all $x \geq 2$.

In general, the number of copies of~$\Z_2$ present in the primary decomposition of the multiplicative group~$\Z_n^\times$ is nearly the same as the number of prime factors of~$n$. Indeed, we will be able to derive our asymptotic formula for $S(x;\Z_2^k)$ quickly from the classical counting function of integers with a given number of prime factors.

\begin{definition} \label{pi_m def}
For all nonnegative integers~$m$, define
\[
\pi_m(x) = \# \{ n \leq x \colon \omega (n) = m \} \quad\text{and}\quad
\pi_m^\star(x) = \# \{ n \leq x \colon 2 \nmid n,\, \omega (n) = m \} \\
\]
\end{definition}

The estimation of $\pi_m(x)$ is classical:

\begin{lemma} \label{pim}
For all positive integers~$m$,
\begin{align*}
\pi_m(x)
=
\frac{1}{(m-1)!}
\frac{x (\log_2 x)^{m-1}}{\log x}
+
O_m \biggl( \frac{x (\log_2 x)^{m-2}}{\log x} \biggr).
\end{align*}
\end{lemma}

\begin{proof}
While this statement can be proved by induction on~$m$, we instead use the Selberg--Sathe asymptotic formula (see for example~\cite[Theorem~6.4]{tenenbaum}), which asserts that
\begin{align*}
\pi_m(x)
=
\frac{1}{(m-1)!}
\frac{x (\log_2 x)^{m-1}}{\log x}
\biggl(
\lambda \biggr( \frac{m-1}{\log_2 x} \biggr)
+
O_m \biggl(\frac{1}{(\log_2 x)^2} \biggr)
\biggr),
\end{align*}
where
\begin{align}
\lambda(z)
=
\frac{1}{\Gamma (1+z)}
\prod_p
\biggl( 1 + \frac{z}{p-1} \biggr)
\biggl(1-\frac{1}{p} \biggr)^z.
\end{align}
Since $\lambda$ is analytic on some neighbourhood of 0, it follows that $\lambda(z) = \lambda(0) + O(z)$ for sufficiently small $z$.
Therefore,
\begin{align*}
\lambda \biggr( \frac{m-1}{\log_2 x} \biggr)
=
1 + O \biggr( \frac{m-1}{\log_2 x} \biggr),
\end{align*}
which completes the proof.
\end{proof}

It turns out that the same asymptotic formula holds when we count only odd integers:

\begin{lemma} \label{cm}
For all positive integers~$m$,
\begin{align*}
\pi_m^\star(x)
=
\frac{1}{(m-1)!}
\frac{x (\log_2 x)^{m-1}}{\log x}
+
O_m \biggl( \frac{x (\log_2 x)^{m-2}}{\log x} \biggr).
\end{align*}
\end{lemma}

\begin{proof}
Every integer counted by $\pi_m(x)$ but not by $\pi_m^\star(x)$ can be written as $n=2^\ell r$ where $r\le x/2^\ell$ and $2\nmid r$ and $\omega(r) = m-1$; consequently,
\begin{align*}
\pi_m(x) - \pi_m^\star(x)
=
\sum_{\ell = 1}^\infty \pi_{m-1}^\star (x/2^\ell)
\leq
\sum_{\ell = 1}^\infty \pi_{m-1} (x/2^\ell).
\end{align*}
The claimed asymptotic formula follows since Lemmas~\ref{function bound} and~\ref{pim} imply that
\[
\sum_{\ell = 1}^\infty \pi_{m-1} (x/2^\ell) \ll_m \frac{x (\log_2 x)^{m-2}}{\log x}. \qedhere
\]
\end{proof}

With the estimation of $\pi_m^\star(x)$ complete, we are ready to establish the asymptotic formula for $S(x;\Z_2^k)$.

\begin{proof}[Proof of Theorem~\ref{main theorem 0}]
Since every odd prime dividing~$n$ produces a primary factor~$\Z_2$ in the multiplicative group~$\Z_n^\times$, we see that $S_0 (x;\Z_2^k)$ is exactly equal to the number of odd integers up to~$x$ with at most $k-1$ distinct (odd) prime factors, which is to say that
\[
S_0(x;\Z_2^k)
=
\sum_{m = 0}^{k-1}
\pi_m^\star (x).
\]
It follows directly from Lemma~\ref{cm} that
\begin{align} \label{S_0 estimate 2}
S_0 (x;\Z_2^k) = \frac{1}{(k-2)!}
\frac{x (\log_2 x)^{k-2}}{\log x}
+
O_k \biggl(
\frac{x (\log_2 x)^{k-3}}{\log x}
\biggr).
\end{align}
Then by equation~\eqref{group by ell} and Lemma~\ref{Sell to S0 lemma},
\begin{align}
S(x;\Z_2^k) &= \sum_{\ell=0}^{\infty} S_\ell(x;\Z_2^k) \notag \\
&= S_0(x,\Z_2^k) + S_0(x/2,\Z_2^k) + S_0 (x/4,\Z_2^{k-1}) + \sum_{\ell = 3}^{\infty} S_0 (x/{2^\ell},\Z_2^{k-2}).
\label{S_l sum evaluation}
\end{align}
By equation~\eqref{S_0 estimate 2}, the first two terms contribute
\begin{align*}
S_0(x, & \Z_2^k) + S_0(x/2,\Z_2^k) \\
&= 
\frac{1}{(k-2)!} \frac{x (\log_2 x)^{k-2}}{\log x}
+ \frac{1}{(k-2)!} \frac{(x/2) (\log_2 (x/2))^{k-2}}{\log (x/2)}
+ O_k \biggl( \frac{x (\log_2 x)^{k-3}}{\log x} \biggr) \\
&= 
\frac{3}{2(k-2)!} \frac{x (\log_2 x)^{k-2}}{\log x}
+ O_k \biggl( \frac{x (\log_2 x)^{k-3}}{\log x} \biggr)
\end{align*}
using Lemma~\ref{log estimate}, while the remaining terms contribute
\begin{align*}
S_0 (x/4,\Z_2^{k-1}) + \sum_{\ell = 3}^{\infty} S_0 (x/{2^\ell},\Z_2^{k-2}) \ll_k \frac{x (\log_2 x)^{k-3}}{\log x} + \frac{x (\log_2 x)^{k-4}}{\log x}
\end{align*}
by equation~\eqref{S_0 estimate 2} again and Lemma~\ref{function bound}. Inserting these last two estimates into equation~\eqref{S_l sum evaluation} completes the proof of the theorem.
\end{proof}

\subsection{Estimation of $S_0(x;\Z_{p^\alpha}^k)$ and $S(x;\Z_{p^\alpha}^k)$}

We now return to the goal of asymptotically evaluating $S_0(x;\Z_{p^\alpha}^k)$ when~$p^\alpha$ is a prime power exceeding~$2$, which is the main technical task of this paper (although we have done much of the technical work already in Section~\ref{downey stuff}). Afterwards we will derive an asymptotic formula for $S(x;\Z_{p^\alpha}^k)$ itself.

\begin{lemma} \label{sum splitting}
For any prime power $p^\alpha \geq 3$ and any~$k\in\N$,
\begin{align} \label{splitting S}
S_0(x;\Z_{p^\alpha}^k) = \sum_{m = 0}^{k-1} D_m(x;p^\alpha),
\end{align}
where $D_m(x;p^\alpha)$ is as in Definition~\ref{D m x p alpha def}.
\end{lemma}

\begin{proof}
By definition, $S_0 (x;\Z_{p^\alpha}^k)$ is the number of integers $N\le x$ not divisible by~$p$ such that $\Z_{p^\alpha}^k \nleq \Z_N^\times$.
Given $N = 2^{\beta} p_1^{\beta_1} ... p_\ell^{\beta_\ell}$ with $p \nmid N$, we have
\begin{align*}
\Z_N^\times
&\cong
\Z_{2^\beta}^\times \times \Z_{p_1^{\beta_1}}^\times \times \cdots \times \Z_{p_\ell^{\beta_\ell}}^\times \\
&\cong \Z_{2^\beta}^\times \times (\Z_{p_1^{\beta_1-1}} \times \Z_{p_1-1}) \times \cdots \times (\Z_{p_\ell^{\beta_\ell-1}} \times \Z_{p_\ell-1}) \\
&\cong ( \Z_{2^\beta}^\times \times \Z_{p_1^{\beta_1-1}} \times \cdots \times \Z_{p_\ell^{\beta_\ell-1}} ) \times (\Z_{p_1-1} \times \cdots \times \Z_{p_\ell-1}).
\end{align*}
Note that we can disregard the $\Z_{2^\beta}^\times$ factor completely, since $\beta=0$ if $p=2$, while $\Z_{p^\alpha}$ cannot be a subgroup of the $2$-group $\Z_{2^\beta}^\times$ if~$p$ is odd. Similarly, we can disregard each factor $\Z_{p_j}^{\beta_j-1}$ since~$p$ does not equal any of the~$p_j$; therefore $\Z_{p^\alpha}^k \nleq \Z_N^\times$ if and only if $\Z_{p^\alpha}^k \nleq \Z_{p_1-1} \times \cdots \times \Z_{p_m-1}$. Finally, each cyclic group $\Z_{p_j-1}$ contains a copy of $\Z_{p^\alpha}$ precisely when $p_j\equiv 1\mod{p^\alpha}$.
Consequently, $\Z_{p^\alpha}^k \nleq \Z_N^\times$ if and only if~$n$ has fewer than~$k$ prime factors~$p_j \equiv 1  \mod{p^\alpha}$, which completes the proof.
\end{proof}

We are now in a position to complete the asymptotic evaluation of $S(x;\Z_{p^\alpha}^k)$, once we define the leading constant $K(p^\alpha,k)$ that appears.

\begin{definition} \label{sum constant}
Given a prime~$p$ and positive integers~$\alpha$ and~$k$, define
\begin{equation*}
K(p^\alpha, k) =
\begin{dcases}
\frac{\delta(\B_{p^\alpha})/(k-1)!}{\Gamma \bigl(1 - 1/{\phi(p^\alpha)} \bigr)} \frac{p^{\alpha+1}-1}{p^{k(\alpha-1) + 1 } (p-1)^k},&
\textrm{if $p \geq 3$,} \\
\frac{\delta(\B_{2^\alpha}) /(k-1)! }{\Gamma \bigl(1 - 1/\phi(2^\alpha) \bigr)} \frac{2^{\alpha+2}-1}{2^{k(\alpha-1) + 2 }},&
\textrm{if $p = 2$ and $\alpha \geq 2$,} \\
1, & \textrm{if $p^\alpha = 2$,}
\end{dcases}
\end{equation*}
where $\B_{p^\alpha}$ is as in Definition~\ref{B p alpha def}.
\end{definition}

\begin{remark}
The quantity $K(2,k)$ in the last case needs to be given some positive value for Theorem~\ref{main theorem 1} to make sense in all possible cases. Its exact value is actually irrelevant, however: the corresponding summand will never be a main term in the sum in equation~\eqref{main theorem 0 eq} (equivalently, $\Z_2^k$ is never a dominant summand of a finite abelian group that is not simply isomorphic to~$\Z_2^k$ itself), and thus the value of $K(2,k)$ never appears in Theorem~\ref{main theorem 2}---the earlier Theorem~\ref{main theorem 0} handles those cases explicitly.
\end{remark}

\begin{remark}
It is possible to write down an exact formula for $\delta(\B_{p^\alpha})$, in terms of either an absolutely convergent infinite product such as~\cite[equation~(1.3)]{chang}, or algebraic invariants of the field $\Q(e^{2\pi i/p^\alpha})$ as in~\cite[Remark~4.2]{chang}. Indeed, the special case $\alpha=1$ already appears as~\cite[Proposition~2.6]{downey}, the proof of wihch uses \cite[Proposition~4.1 and Remark~3.5]{chang}.
\end{remark}

\begin{proposition} \label{S-sum estimate}
For any prime power $p^\alpha\ge3$, any $k \in \N$, and any $x\ge3$,
\begin{equation*}
S(x;\Z_{p^\alpha}^k)
= K(p^\alpha, k)
\frac{ x(\log_2 x)^{k-1}}{ (\log x)^{1/\phi(p^\alpha)}}
+ O_{p, \alpha , k} \biggl ( \frac{x ( \log_2 x)^{k-2}}{(\log x)^{1/ \phi(p^\alpha)}} \biggr ).
\end{equation*}
\end{proposition}

\begin{proof}
We first note that if we set $\alpha^* = \alpha$ when $p$ is odd and $\alpha^* = \alpha+1$ when $p=2$, it follows from equation~\eqref{group by ell} and Lemma~\ref{Sell to S0 lemma} that
\begin{equation} \label{S-sum first step}
S(x;\Z_{p^\alpha}^k)
= \sum_{\ell=0}^{\infty} S_\ell(x; \Z_{p^\alpha}^k)
= \sum_{\ell=0}^{\alpha^*} S_0(x/p^\ell; \Z_{p^\alpha}^k)
+ \sum_{\ell = \alpha^*+1}^\infty S_0(x/p^\ell; \Z_{p^\alpha}^{k-1}).
\end{equation}
By Lemma~\ref{D-sum estimate new},
we can rewrite Lemma~\ref{sum splitting} as
\begin{align} \label{S_0 estimate}
S_0 (x; \Z_{p^\alpha}^k)
&= \sum_{m = 0}^{k-1} D_m(x;p^\alpha)&
\nonumber
\\
&=
\frac{\delta(\B_{p^\alpha}) /(k-1)! }{\Gamma (1 - 1/\phi(p^\alpha))}
\frac{x(\log_2 x)^{k-1}}{\phi(p^\alpha)^{k-1} (\log x)^{1/\phi(p^\alpha)}}
+
O_{p, \alpha , k}
\biggl (
\frac{x(\log_2 x)^{k-2}}{ (\log x)^{1/\phi(p^\alpha)}}
\biggr ).
\end{align}
In particular, $S_0 (x;  \Z_{p^\alpha}^{k-1}) \ll_{p,\alpha,k} {x (\log_2 x)^{k-2}} / {(\log x)^{1/\phi(p^\alpha)}}$, and therefore Lemma~\ref{function bound} implies
\[
\sum_{\ell = \alpha^*+1}^\infty S_0(x/p^\ell;  \Z_{p^\alpha}^{k-1}) \ll_{p,\alpha,k} \frac{x(\log_2 x)^{k-2}}{ (\log x)^{1/\phi(p^\alpha)}}.
\]
It therefore remains to evaluate the first sum on the right-hand side of equation~\eqref{S-sum first step}. By equation~\eqref{S_0 estimate} and Lemma~\ref{log estimate},
\begin{align*}
\sum_{\ell=0}^{\alpha^*} & S_0(x/p^\ell;  \Z_{p^\alpha}^{k}) \\
&=
\frac{x \delta(\B_{p^\alpha}) /(k-1)! }{\phi(p^\alpha)^{k-1}\Gamma (1 - 1/\phi(p^\alpha))}
\sum_{\ell=0}^{\alpha^*}
\biggl (
\frac{(\log_2 (x/p^\ell))^{k-1}}{p^\ell(\log (x/p^\ell))^{1/ \phi(p^\alpha)}}
+
O_{p, \alpha , k}
\biggl (
\frac{(\log_2 (x/p^\ell))^{k-2}}{(\log (x/p^\ell))^{1/ \phi(p^\alpha)}}
\biggr )
\biggr ) \\
&=
\frac{x \delta(\B_{p^\alpha}) /(k-1)! }{\phi(p^\alpha)^{k-1}\Gamma (1 - 1/\phi(p^\alpha))}
\sum_{\ell=0}^{\alpha^*}
\biggl (
\frac{(\log_2 x)^{k-1}}{p^\ell(\log x)^{1/ \phi(p^\alpha)}}
+
O_{p, \alpha , k}
\biggl (
\frac{(\log_2 x)^{k-2}}{(\log x)^{1/ \phi(p^\alpha)}}
\biggr ) \biggr) \\
&=
\frac{x \delta(\B_{p^\alpha}) /(k-1)! }{\phi(p^\alpha)^{k-1}\Gamma (1 - 1/\phi(p^\alpha))}
\frac{(\log_2 x)^{k-1}}{(\log x)^{1/\phi(p^\alpha)}} \sum_{\ell = 0}^{\alpha^*} \frac{1}{p^\ell}
+ O_{p, \alpha , k}
\biggl (
\frac{x(\log_2 x)^{k-2}}{(\log x)^{1/ \phi(p^\alpha)}}
\biggr ).
\end{align*}
The application of Lemma~\ref{log estimate} requires that $\sqrt x \ge p^{\alpha^*}$, but because the implicit constant in the error term is allowed to depend on~$p$ and~$\alpha$, the resulting asymptotic formula is valid for all $x\ge3$.
Since
\begin{align*}
\frac{1}{\phi(p^\alpha)^{k-1}} \sum_{\ell=0}^{\alpha^*} \frac{1}{p^\ell}
= \frac{1}{p^{(\alpha -1)(k-1)} (p-1)^{k-1}} \frac{p^{\alpha^*+1}-1}{p^{\alpha^*} (p-1)}
= \begin{dcases}
\frac{p^{\alpha+1}-1}{p^{k(\alpha-1) + 1 } (p-1)^k}, & \textrm{ if $p \geq 3$}, \\
\frac{2^{\alpha+2}-1}{2^{k(\alpha-1) + 2 }}, & \textrm{ if $p = 2$},
\end{dcases}
\end{align*}
the proposition now follows by Definition~\ref{sum constant}.
\end{proof}

Between Theorem~\ref{main theorem 0} and Proposition~\ref{S-sum estimate},
we have now found asymptotic formulas for $S(x;\Z_{p^\alpha}^k)$ for all gathered summands $\Z_{p^\alpha}^k$.
To complete the proof of Theorem~\ref{main theorem 2} itself which concerns the general counting function $S(x;G)$, we need to understand how the various gathered summands of~$G$ interact.

\section{General finite abelian groups} \label{general section}

If the cardinalities of two finite abelian groups~$G_1$ and~$G_2$ are relatively prime, then an abelian group contains a copy of $G_1\times G_2$ if and only if it contains a copy of~$G_1$ and a copy of~$G_2$. Consequently, $S(x;G_1\times G_2) \le S(x;G_1) + S(x;G_2)$ when $(\#G_1,\#G_2)=1$, and the only reason this inequality is not an equality is because the right-hand side double-counts multiplicative groups that contain neither~$G_1$ nor~$G_2$. This inequality extends immediately to any number of factors with relatively prime cardinalities.

In general, however, a finite abelian group can have subgroups isomorphic to both~$G_1$ and~$G_2$ without having a subgroup isomorphic to $G_1\times G_2$. For example, $\Z_{16}^\times \cong \Z_4 \times \Z_2$ has subgroups isomorphic to $\Z_4$ and $\Z_2^2$, but does not have a subgroup isomorphic to $\Z_4\times \Z_2^2$.
It follows that $S(x;G_1\times G_2)$ is not directly comparable to $S(x,G_1) + S(x;G_2)$ in general. However, this difficulty can be overcome with a bit of deliberate attention.

In the next section, we estimate $S(x;H_1)$ and $S(x;H_1\times H_2)$ when~$H_1$ and~$H_2$ are both finite abelian $p$-groups. These are the last analytic tools we need to prove our main result, Theorem~\ref{main theorem 2}, in the final section.

\subsection{Finite abelian $p$-groups}

Our first lemma uses a convenient group-theoretic observation to show that the main contribution to $S(x;H)$, where~$H$ is a finite abelian $p$-group, comes from the dominant summand of~$H$ (which in this case is the gathered summand of~$H$ corresponding to the largest power of~$p$).

\begin{lemma} \label{S(x;H)}
Let $H \cong \Z_{p^{\alpha_1}}^{k_1} \times \cdots \times \Z_{p^{\alpha_\ell}}^{k_\ell}$ be a finite abelian $p$-group
with $\alpha_1 < \cdots < \alpha_\ell$, and set $m = k_1 + k_2 + \cdots + k_\ell$.
Then
\begin{equation}
S(x; H) = S(x;\Z_{p^{\alpha_\ell}}^{k_\ell}) + O_H \biggl( \frac{x (\log_2 x)^{m-1}}{(\log x)^{1/\phi(p^{\alpha_\ell-1})}}  \biggr) .          \nonumber
\end{equation}
\end{lemma}

\begin{proof}
Note that $S(x;\Z_{p^{\alpha_\ell}}^{k_\ell}) \leq S(x; H)$ since $\Z_{p^{\alpha_\ell}}^{k_\ell} \leq H$.
Additionally, $H \leq \Z_{p^{\alpha_\ell-1}}^{m-k_\ell} \times \Z_{p^{\alpha_\ell}}^{k_\ell}$, and if an abelian group contains both $\Z_{p^{\alpha_\ell-1}}^{m}$ and $\Z_{p^{\alpha_\ell}}^{k_\ell}$ then it also contains $\Z_{p^{\alpha_\ell-1}}^{m-k_\ell} \times \Z_{p^{\alpha_\ell}}^{k_\ell}$.
Therefore $S(x;H) \leq  S(x;\Z_{p^{\alpha_\ell-1}}^{m-k_\ell} \times \Z_{p^{\alpha_\ell}}^{k_\ell}) \leq S(x;\Z_{p^{\alpha_\ell-1}}^m) + S(x;\Z_{p^{\alpha_\ell}}^{k_\ell})$.
Hence $S(x; H) = S(x;\Z_{p^{\alpha_\ell}}^{k_\ell}) + O \bigl( S(x;\Z_{p^{\alpha_\ell-1}}^m) \bigr)$, which completes the proof by Proposition~\ref{S-sum estimate}.
\end{proof}

\begin{corollary} \label{S-sum corollary}
Let~$G$ be a finite abelian group.
If $\Z_{p^\alpha}^k$ is a dominant summand of~$G$, and~$H\le G$ is a $q$-group with $\Z_{p^\alpha}^k \nleq H$, then
\begin{equation} \label{S sum 1}
S(x;H) \ll_G \frac{ x(\log_2 x)^{k-2}}{(\log x)^{1/\phi(p^\alpha)}}.
\end{equation}
\end{corollary}

\begin{proof}
If $\Z_{q^\beta}^\ell$ is the dominant summand of~$H$, then Lemma~\ref{S(x;H)} implies that for some integer~$m$ depending on~$H$,
\begin{equation} \label{S sum 2}
S(x;H) = S(x;\Z_{q^\beta}^\ell) + O_H \biggl( \frac{x (\log_2 x)^{m-1}}{(\log x)^{1/\phi(q^{\beta-1})}}  \biggr) \ll_G \frac{ x(\log_2 x)^{\ell-1}}{(\log x)^{1/\phi(q^\beta)}}
\end{equation}
by Proposition~\ref{S-sum estimate} (or Theorem~\ref{main theorem 0} if $q^\beta=2$). Note that $\Z_{q^\beta}^\ell$ is a gathered summand, but not a dominant summand, of~$G$.
By Definition~\eqref{dominant def}, either $\phi(q^\beta) < \phi(p^\alpha)$, or else $\phi(q^\beta) = \phi(p^\alpha)$ but $\ell<k$. In the second case, the bound~\eqref{S sum 2} implies the bound~\eqref{S sum 1} immediately since $\ell\le k-1$; in the first case, the bound~\eqref{S sum 2} implies the bound~\eqref{S sum 1}, regardless of the values of~$k$ and~$\ell$, because ${(\log x)^{1/\phi(p^\alpha)}} \ll {(\log x)^{1/\phi(q^\beta)}}$.
\end{proof}

The inequality in the first paragraph of Section~\ref{general section} is useful, but in one circumstance we will need not just an inequality but an asymptotic formula. We accomplish this, when $(\#H_1,\#H_2)=1$, by writing
\begin{equation} \label{incl excl}
S(x;G) = S(x;H_1) + S(x;H_2) - S(x;H_1,H_2),
\end{equation}
where (to compensate for the aforementioned double-counting) we have defined
\begin{equation*}
S (x; H_1 , H_2) = \sum_{\substack{n \leq x \\ H_1 \nleq \Z_n^\times \\ H_2 \nleq \Z_n^\times}} 1.
\end{equation*}

\begin{lemma} \label{S(x;H1H2)}
Let~$H_1$ be a finite abelian $p_1$-group and~$H_2$ a finite abelian $p_2$-group with $p_1\ne p_2$, and set $H=H_1\times H_2$. For $i=1,2$ write
\begin{equation*}
H_i \cong \Z_{p_i^{\alpha_{1,i}}}^{k_{1,i}} \times \Z_{p_i^{\alpha_{2,i}}}^{k_{2,i}} \times \cdots \times \Z_{p_i^{\alpha_{\ell_i, i}}}^{k_{\ell_i, i}}
\end{equation*}
with $\alpha_{1, i} < \cdots < \alpha_{\ell_i, i}$, and set $\beta_i = \alpha_{\ell_i,i}$. Then
\begin{equation} \label{rho error}
S(x;H) =  S(x; H_1) + S(x; H_2)  +
O_H \biggl(\frac{x (\log_2 x)^{m_1 + m_2-2}}{(\log x)^\rho}
\biggr),
\end{equation}
where $m_i = k_{1,i} + \cdots + k_{\ell_i, i}$ and $\rho = 1/\phi(p_1^{\beta_1}) + 1/\phi(p_2^{\beta_2}) - 1/\phi(p_1^{\beta_1} p_2^{\beta_2})$.
In particular, if $p_1^{\beta_1} \ge3$ then
\begin{equation} \label{-1 error}
S(x;H) = S(x; H_1) + S(x; H_2) + O_H \biggl( \frac{x (\log_2 x)^{-1}}{(\log x)^{1/\phi(p_1^{\beta_1})}} \biggr).
\end{equation}
\end{lemma}

\begin{proof}
By the inclusion--exclusion formula~\eqref{incl excl}, for the first claim it suffices to show that $S(x; H_1, H_2)$ is bounded by the error term in equation~\eqref{rho error}. Moreover, since $H_i \le \Z_{p_i^{\beta_i}}^{m_i}$, it suffices to establish such a bound for $S(x; \Z_{p_1^{\beta_1}}^{m_1}, \Z_{p_2^{\beta_2}}^{m_2})$.

Let $\B_H$ be the set of reduced residue classes modulo $p_1^{\beta_1} p_2^{\beta_2}$ such that $c \in \B_H$ precisely when $c \not \equiv 1 \mod{p_1^{\beta_1}}$ and $c \not \equiv 1 \mod{ p_2^{\beta_2}}$. By an argument similar to the proof of Lemma~\ref{sum splitting}, no prime in~$\B_H$ can contribute any factors of $\Z_{p_1^{\beta_1}}$ or $\Z_{p_2^{\beta_2}}$ to a multiplicative group.
It follows that if both $\Z_{p_1^{\beta_1}}^{m_1} \nleq \Z_n^\times$ and $\Z_{p_2^{\beta_2}}^{m_2} \nleq \Z_n^\times$, then~$n$ has at most $m_1+m_2-2$ prime factors $q \notin \B_H$.
Since $\B_H$ is the union of $\tau = \phi(p_1^{\beta_1}p_2^{\beta_2}) - \phi(p_1^{\beta_1}) - \phi(p_2^{\beta_2}) + 1$ distinct reduced residue classes modulo~$p_1^{\beta_1} p_2^{\beta_2}$, Lemma~\ref{C(x,T) new} implies that
\begin{align*}
S(x; \Z_{p_1^{\beta_1}}^{m_1}, \Z_{p_2^{\beta_2}}^{m_2}) \ll_H
\frac{x (\log_2 x)^{m_1 + m_2-2}}{(\log x)^{1-\tau/\phi(p_1^{\beta_1} p_2^{\beta_2})}},
\end{align*}
which is the same as the error expression in equation~\eqref{rho error}.

In addition, $\rho - 1/\phi(p_1^{\beta_1}) = (1 - 1/\phi(p_1^{\beta_1}))/\phi(p_2^{\beta_2}) > 0$ when $p_1^{\beta_1} \ge3$, which means that equation~\eqref{rho error} implies equation~\eqref{-1 error}.
\end{proof}

\subsection{The general case}

We now have all the tools needed to complete our estimation of $S(x;G)$ and thereby establish Theorem~\ref{main theorem 2}. The statement of that theorem, however, includes the assumption (as mentioned in Remark~\ref{dominant remark}) that a finite abelian group can have at most two dominant summands; we now justify that assumption with the following simple elementary number theory argument.

\begin{lemma} \label{puzzle}
There can be at most two prime powers that share the same $\phi$-value. Furthermore, when this occurs, one of the prime powers must be an actual prime.
\end{lemma}

\begin{proof}
Suppose there exist prime powers $p_1^{\alpha_1} \neq p_2^{\alpha_2}$ such that $\phi (p_1 ^{\alpha_1}) = \phi (p_2 ^{\alpha_2})$.
Then $p_1 \neq p_2$, and we may assume without loss of generality that $p_1 > p_2$.
It thus follows that $\alpha_1 = 1$ as~$p_1$ divides neither~$p_2$ nor~$p_2 - 1$.
Then for any $p^\alpha$ such that $\phi(p^\alpha) = \phi (p_1 ^{\alpha_1}) = \phi (p_2 ^{\alpha_2})$, either $p^\alpha = p_2^{\alpha_2}$, or it again follows that $\alpha = 1$ and thus that $p = p_1$.
\end{proof}

\begin{proof}[Proof of Theorem~\ref{main theorem 2}]
For any way of writing
\begin{equation*}
G \cong G_1 \times G_2 \times \cdot \cdot \cdot \times G_\ell
\end{equation*}
where the $\# G_i$ are pairwise relatively prime, we have
\[
S(x;G_1) \le S(x;G) \le S(x;G_1) + \cdots + S(x;G_\ell)
\]
(the first inequality holds because $G_1 \le G$, while the second inequality holds by the union bound described in the first paragraph of Section~\ref{general section}); in particular,
\begin{equation} \label{asy}
S(x;G) = S(x;G_1) + O\bigl( S(x;G_2) + \cdots + S(x;G_\ell) \bigr).
\end{equation}

Suppose first that~$G$ has a unique dominant summand $\Z_{p^\alpha}^k$. We take~$G_1$ to be the $p$-Sylow subgroup of~$G$ (the largest $p$-group contained in~$G$) and $G_2,\dots,G_\ell$ to be the Sylow subgroups of~$G$ corresponding to the other prime factors of~$\#G$. Using Lemma~\ref{S(x;H)} for the main term and Corollary~\ref{S-sum corollary} for the error terms, equation~\eqref{asy} becomes
\begin{align*}
S(x;G)
&= \biggl( S(x;\Z_{p^\alpha}^k)
+O_{G_1}\biggl(\frac{x(\log_2 x)^m}{(\log x)^{1/ \phi(p^{\alpha-1})}} \biggr) \biggr)
+ O_G \biggl( \sum_{j=2}^\ell \frac{x (\log_2 x)^{k-2}}{(\log x)^{1/\phi(p^{\alpha})}} \biggr) \\
&= K(p^{\alpha}, k) \frac{x (\log_2 x)^{k-1}}{(\log x)^{1/\phi(p^{\alpha})}}
+ O_{G} \biggl(
\frac{x (\log_2 x)^{k-2}}{(\log x)^{1/\phi(p^{\alpha})}}
\biggr)
\end{align*}
by Proposition~\ref{S-sum estimate} (note that $p^\alpha\ge3$ since $G\not\cong\Z_2^k$ by assumption), which establishes the theorem in this case.

Suppose now that~$G$ has two dominant summands $\Z_{p^\alpha}^k$ and $\Z_{q^\beta}^k$. We take $G_1 \cong H_1 \times H_2$, where~$H_1$ is the $p$-Sylow subgroup of~$G$ and $H_2$ is the $q$-Sylow subgroup of~$G$, and $G_2,\dots,G_\ell$ to be the Sylow subgroups of~$G$ corresponding to the other prime factors of~$\#G$.
Using Lemma~\ref{S(x;H1H2)} for the main term and Corollary~\ref{S-sum corollary} for the error terms, equation~\eqref{asy} becomes
\begin{align*}
S(x;G_1) &= S(x; H_1) + S(x; H_2) + O_G \biggl( \frac{x (\log_2 x)^{-1}}{(\log x)^{1/\phi(p^\alpha)}} \biggr) + O_G \biggl( \sum_{j=2}^\ell \frac{x (\log_2 x)^{k-2}}{(\log x)^{1/\phi(p^{\alpha})}} \biggr) \\
&= K(p^{\alpha}, k) \frac{x (\log_2 x)^{k-1}}{(\log x)^{1/\phi(p^{\alpha})}} + K(q^\beta, k) \frac{x (\log_2 x)^{k-1}}{(\log x)^{1/\phi(q^\beta)}}
+ O_{G} \biggl(
\frac{x (\log_2 x)^{k-2}}{(\log x)^{1/\phi(p^{\alpha})}}
\biggr)
\end{align*}
again by Proposition~\ref{S-sum estimate}, which establishes the theorem in this case since $\phi(p^\alpha) = \phi(q^\beta)$.
\end{proof}

\section*{Acknowledgments}

The second author was supported in part by a Natural Sciences and Engineering Council of Canada Discovery Grant.

\bibliographystyle{abbrv}
\bibliography{bibliography}

\end{document}